\documentclass{article}
\usepackage{graphicx}
\usepackage{tikz}
\usepackage[algoruled,linesnumbered]{algorithm2e}
\usepackage{amsmath, amssymb, amsthm}
\usepackage{verbatim}
\usepackage{enumerate}
\usepackage{url}
\usepackage{subcaption}

\pgfdeclarelayer{edgelayer}
\pgfdeclarelayer{nodelayer}
\pgfsetlayers{edgelayer,nodelayer,main}

\newtheorem{theorem}{Theorem}[section]
\newtheorem{definition}[theorem]{Definition}

\newtheorem{proposition}[theorem]{Proposition}
\newtheorem{corollary}[theorem]{Corollary}
\newtheorem{problem}[theorem]{Problem}

\providecommand{\pe}{\mathop{\rm pe}\nolimits}

\providecommand{\Bn}{B^n}

\begin{document}

\oddsidemargin 16.5mm
\evensidemargin 16.5mm

\thispagestyle{plain}

\title{Daisy cubes: a characterization and a generalization}
\author{Andrej Taranenko\\
\small Faculty of Natural Sciences and Mathematics, University of Maribor,\\[-0.8ex]
\small Koro\v{s}ka cesta 160, SI-2000 Maribor, Slovenia\\
\small and\\
\small Institute of Mathematics, Physics and Mechanics \\[-0.8ex]
\small Jadranska 19, SI-1000 Ljubljana, Slovenia\\[-0.8ex]
\ \\
\small \texttt{andrej.taranenko@um.si}\\
}

\maketitle

\begin{abstract}
Daisy cubes are a recently introduced class of isometric subgraphs of hypercubes $Q_n$. They are induced with intervals between chosen vertices of $Q_n$ and the vertex $0^n\in V(Q_n)$. In this paper we characterize daisy cubes in terms of an expansion procedure thus answering an open problem proposed by Klav\v{z}ar and Mollard, 2018, in the introductory paper of daisy cubes \cite{KlaMol-18}. To obtain such a characterization several interesting properties of daisy cubes are presented. For a given graph $G$ isomorphic to a daisy cube, but without the corresponding embedding into a hypercube, we present an algorithm which finds a proper embedding of $G$ into a hypercube in $O(mn)$ time. Finally, daisy graphs of a rooted graph are introduced and shown to be a generalization of daisy cubes.
\end{abstract}

\noindent \emph{Keywords.} {daisy cubes, characterization, generalization.}
\\ \noindent \emph{Mathematics Subject Classifications.} 05C75, 05C85, 68R10.


\section{Introduction}
Daisy cubes have been recently introduced by Klav\v{z}ar and Mollard \cite{KlaMol-18}. This is a new class of partial cubes that contains other well known families of cube-like graphs, e.g. Fibonacci cubes (\cite{K2012}) and Lucas cubes (\cite{MCS2001,T2013}). Since the introduction other results on daisy cubes have already appeared. Providing a connection to chemical graph theory, \v{Z}igert Pleter\v{s}ek has shown in  \cite{ZP2018} that the resonance graphs of kinky benzenoid systems are daisy cubes. In \cite{V2019} Vesel has shown that the cube-complement of a daisy cube is also a daisy cube.

In the introductory paper \cite{KlaMol-18} several nice properties and open problems about daisy cubes are presented. Among others the following problem.

\begin{problem}\cite[Problem 5.1]{KlaMol-18}\label{KlaMolP51}
Do daisy cubes admit a characterization in terms of an expansion procedure?
\end{problem}

In this paper we show that the answer to Problem \ref{KlaMolP51} is affirmative. We proceed as follows. In the continuation of this section we provide basic definitions and results needed throughout the paper. In section \ref{DaisyCubes}, in the process of providing a characterization of daisy cubes in terms of an expansion procedure, we prove several other interesting properties of daisy cubes. We continue with section \ref{findingLabellings} by providing an algorithm which for a given graph $G$ isomorphic to a daisy cube finds a proper labelling of the vertices of $G$ and thus an isometric embedding into the corresponding hypercube. The paper is concluded with a generalization of daisy cubes and some open problems.

Let $B=\{0,1\}$. A word $u$ of length $n$ over $B$ is called a binary string of length $n$ and will be denoted by $u=(u_1, u_2, \ldots, u_n) \in \Bn$, or shorter $u_1 u_2 \ldots u_n$. We will use the power notation for the concatenation of bits, for instance $0^n = 0\ldots 0\in \Bn$.

The \emph{$n$-cube} $Q_n$ is the graph with the vertex set $\Bn$, where two vertices are adjacent whenever the two binary strings differ in exactly one position. 

For two vertices $u$ and $v$ of a graph $G$, the interval $I_G(u, v)$ between $u$ and $v$ in $G$ is the set of all vertices lying on some shortest $u, v$-path, that is, $I_G(u, v) = \{w \in V(G)\ |\  d(u, v) = d(u, w) + d(w, v)\}$. The index $G$ may be omitted where the graph will be clear from the context. A subgraph $H$ of a graph $G$ is isometric if $d_H (u, v) = d_G(u, v)$ holds for any $u, v \in V(H)$. Isometric graphs of hypercubes are called \emph{partial cubes}.

A \emph{median} of a triple of vertices $u$, $v$ and $w$ of a graph $G$ is a vertex $z\in V(G)$ that lies on a shortest $u, v$-path, on a shortest $u, w$-path, and on a shortest $v, w$-path. If $G$ is connected and every triple of vertices admits a unique median, then $G$ is a \emph{median graph}. It is well known that median graphs are partial cubes (cf. \cite{HIK2011knjiga}).

Let $G$ be a connected graph with $e=xy$ and $f=uv$ two edges in $G$. We say that $e$ is in relation $\Theta$ to $f$ if $d(x,u)+d(y,v) \not = d(x,v) + d(y,u)$. $\Theta$ is reflexive and symmetric, but need not be transitive. We denote its transitive closure by $\Theta^{*}$. It was proved in \cite{W1984} that $G$ is a partial cube if and only if $G$ is bipartite and $\Theta = \Theta^{*}$.

For $X\subseteq V(G)$ we denote the subgraph of $G$ induced by the set $X$ with $\langle X\rangle_G$ or simply as $\langle X\rangle$, when $G$ is clear from the context.

A subgraph $H$ of a graph $G$ is called \emph{convex}, if it is connected and if every shortest path of $G$ between two vertices of $H$ is completely contained in $H$.

For an edge $ab$ of graph $G$ we define:
\begin{itemize}
	\item $W_{ab} = \{w \in V(G)\ |\ d(a,w) < d(b,w) \}$
	\item $W_{ba} = \{w \in V(G)\ |\ d(b,w) < d(a,w) \}$
	\item $F_{ab} = \{xy \in E(G)\ |\ x \in W_{ab} \text{ and } y \in W_{ba} \}$
	\item $U_{ab} = \{w \in W_{ab}\ |\ w \text{ is the end vertex of an edge in } F_{ab} \}$
	\item $U_{ba} = \{w \in W_{ba}\ |\ w \text{ is the end vertex of an edge in } F_{ab}\}$
\end{itemize}

Let $ab$ be an edge of a partial cube $G$ for which $U_{ab} = W_{ab}$. Then $\langle W_{ab} \rangle$ is called a \emph{peripheral subgraph} of $G$. A $\Theta$-class $E$ of a median graph $G$ is called \emph{peripheral}, if at least one of $\langle W_{ab} \rangle$ and $\langle W_{ba} \rangle$ is peripheral for $ab \in E$. $E$ is \emph{internal} if it is not peripheral.

In \cite{KlaMol-18} daisy cubes were defined as follows. Let $\leq$ be a partial order on $\Bn$ defined in the following way: $u_1 u_2 \ldots u_n \leq v_1 v_2 \ldots v_n$ if $u_i \leq v_i$ is true for all $i \in \{1, 2, \ldots, n\}$. For $X\subseteq \Bn$ the graph $Q_n(X)$ is defined as the subgraph of $Q_n$ where $$Q_n(X)=\langle \{u \in \Bn \ |\ u \leq x \text{ for some }x\in X\} \rangle.$$ The graph $Q_n(X)$ is called \emph{a daisy cube (generated by $X$)}.

As noted in \cite{KlaMol-18}, if $x,y\in X$ are such that $y\leq x$, then $Q_n(X) = Q_n(X \setminus \{y\})$. 

\begin{proposition}\label{DCisPC}\cite{KlaMol-18}
If $X\subseteq \Bn$, then $Q_n(X)$ is a partial cube.
\end{proposition}

Several characterizations of partial cubes and median graphs are known. Before we state some, we will need the notion of expansion.

Suppose $V(G)=V_1 \cup V_2$, where $V_1 \cap V_2 \not = \emptyset$, both $\langle V_1 \rangle$ and $\langle V_2\rangle$ are isometric subgraphs of $G$, and there exists no edge of $G$ with one endpoint in $V_1\setminus V_2$ and the other in $V_2 \setminus V_1$. An \emph{expansion} of $G$ with respect to $V_1$ and $V_2$ is a graph $G'$ obtained from $G$ by the following steps:
\begin{enumerate}[(i)]
\item Replace each $v\in V_1 \cap V_2$ with vertices $v_1, v_2$, and add the edge $v_1v_2$.
\item Add edges between $v_1$ and all neighbours of $v$ in $V_1\setminus V_2$, also add edges between $v_2$ and all neighbours of $v$ in $V_2\setminus V_1$.
\item Insert the edges $v_1 u_1$ and $v_2 u_2$ if $u,v \in V_1 \cap V_2$ are adjacent in $G$.
\end{enumerate}

An expansion is called a \emph{connected expansion} if $\langle V_1 \cap V_2 \rangle$ is a connected subgraph of $G$. An expansion is called \emph{a convex expansion} if $\langle V_1 \cap V_2 \rangle$ is a convex subgraph of $G$. An expansion is \emph{peripheral} if $V_1=V(G)$. We will denote peripheral expansions by $\pe(G;V_2)$ and call them peripheral expansions of $G$ with respect to $V_2$.

The inverse operation of an expansion is called \emph{a contraction}. Note, a contraction of a partial cube is obtained by contracting the edges of a given $\Theta$-class.

The following characterization of partial cubes is well known.

\begin{theorem}\cite{chepoi-88}\label{charPCChepoi}
A graph $G$ is a partial cube if and only if $G$ can be obtained from the one-vertex graph by a sequence of expansions.
\end{theorem}

Median graphs are also characterized in a manner similar to Theorem  \ref{charPCChepoi}. This characterization is known as the Mulder's Convex Expansion Theorem:

\begin{theorem}\cite[Theorem 12.8]{HIK2011knjiga}\label{muldersExpThm}
A graph $G$ is a median graph if and only if it can be obtained from the one-vertex graph by a sequence of convex expansions.
\end{theorem}

In the next section we will characterize daisy cubes in the language of expansions, similar to Theorem \ref{charPCChepoi} and Theorem \ref{muldersExpThm}. For this, we will need the following result.

\begin{proposition}\cite{KlaMol-18}\label{contractionIsDC}
If $G = Q_n(X)$ is a daisy cube, then a contraction of $G$ is a daisy cube.
\end{proposition}

\section{Characterization}\label{DaisyCubes}

Let $G$ be a graph isomorphic to a daisy cube $Q_h(X)$. There can be more than one isometric embedding of $G$ into the hypercube $Q_h$ (each such embedding assigns corresponding  labelling to vertices of $G$). Let $X_G\subseteq B^h$ be the set of labels of the vertices of $G$ assigned by an isometric embedding of $G$ into $Q_h$. We say that $G$ has a \emph{proper labelling} if $G$ is isomorphic to $Q_h(X_G)$. Otherwise the labelling is \emph{improper}. 

In Figure \ref{figPropera} the graph $G$ is isomorphic to the daisy cube $Q_3(\{011, 100\})$. In Figure \ref{figProperb} we see two other embeddings of the graph $G$ into the hypercube $Q_3$. The embedding shown with thick dashed lines is defined by: $u_1 \mapsto 000$, $u_2 \mapsto 001$, $u_3 \mapsto 101$, $u_4 \mapsto 100$ and $u_5 \mapsto 010$. This embedding assigns a proper labelling to the vertices of $G$. The embedding shown with thick solid lines is defined by: $u_1 \mapsto 110$, $u_2 \mapsto 010$, $u_3 \mapsto 011$, $u_4 \mapsto 111$ and $u_5 \mapsto 100$, corresponds to an improper labelling (e.g. $101 \leq 111$ and $101$ is not a label of any vertex in $G$).

\begin{figure}[!ht]
	\centering
\begin{subfigure}[t]{0.4\linewidth}
\begin{tikzpicture}
\tikzstyle{rn}=[circle,fill=black,draw, inner sep=0pt, minimum size=5pt]
\tikzstyle{every node}=[font=\footnotesize]  
		\node [style=rn] (0) at (0, 0) [label=below:$u_1$] {};

		\node [style=rn] (3) at (0, 1) [label=above:$u_2$] {};
		\node [style=rn] (4) at (1, 1) [label=above:$u_3$] {};
		\node [style=rn] (5) at (1, 0) [label=below:$u_4$] {};
		\node [style=rn] (7) at (-1, 0) [label=below:$u_5$] {};
		
		 \draw(0)--(5);
		 \draw(7)--(0)--(3);		 
		 \draw(3)--(4);
		 \draw(5)--(4);	
			 
\end{tikzpicture}
		\caption{A graph $G$ isomorphic to a daisy cube.}
		\label{figPropera}
		\end{subfigure}\qquad
		\begin{subfigure}[t]{0.4\linewidth}
\begin{tikzpicture}
\tikzstyle{rn}=[circle,fill=black,draw, inner sep=0pt, minimum size=5pt]
\tikzstyle{every node}=[font=\footnotesize]  
		\node [style=rn] (0) at (0, 0) [label=below:$000$] {};
		\node [style=rn] (1) at (0, 1) [label=above:$001$] {};
		\node [style=rn] (2) at (1, 1) [label=above:$011$] {};
		\node [style=rn] (3) at (1, 0) [label=below:$010$] {};
		
		\node [style=rn] (4) at (-1, -1) [label=below:$100$] {};
		\node [style=rn] (5) at (-1, 2) [label=above:$101$] {};
		\node [style=rn] (6) at (2, 2) [label=above:$111$] {};
		\node [style=rn] (7) at (2, -1) [label=below:$110$] {};

 		 \draw[dotted](0)--(1)--(2)--(3)--(0); 		 
 		 \draw[dotted](4)--(5)--(6)--(7)--(4);		 
		 \draw[dotted](0)--(4);
		 \draw[dotted](1)--(5);
		 \draw[dotted](2)--(6);
		 \draw[dotted](3)--(7);	
		 
		 \draw[dashed,line width=2pt](0)--(1)--(5)--(4)--(0);
		 \draw[dashed,line width=2pt](0)--(3);
		 
		 \draw[solid,line width=2pt](2)--(3)--(7)--(6)--(2);
		 \draw[solid,line width=2pt](4)--(7);
			 
\end{tikzpicture}
		\caption{Two embeddings of $G$ into the hypercube $Q_3$.}
		\label{figProperb}
		\end{subfigure}
		\caption{Example of a proper and a improper isometric embedding.}
		\label{figProper}
\end{figure}
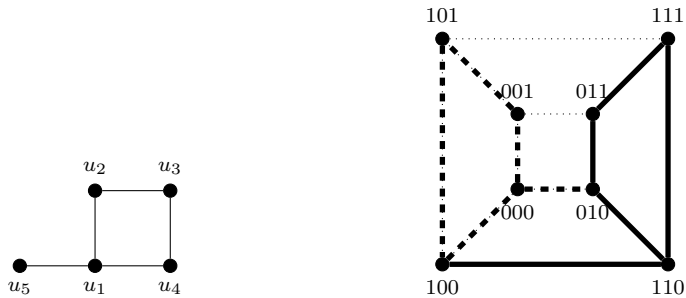

Since the embedding of the daisy cube $Q_h(X)$ induced by the definition of a daisy cube gives a proper labelling of its vertices, from here on we can always assume that labellings considered in graphs isomorphic to a daisy cube are proper labellings. Later on, we will also provide an algorithms which for a given unlabelled graph $G$ isomorphic to some daisy cube assigns a proper labelling to the vertices of $G$.

We now continue with results leading to a characterization of daisy cubes.

\begin{proposition}\label{thetaIsPeripheral}
Every $\Theta$-class of a daisy cube $G$ is peripheral.
\end{proposition}

\begin{proof}
Let $G$ be a daisy cube. Consider an arbitrary edge $e=ab$ of $G$. Its endpoints differ in exactly one position, say position $i$. We choose the notation such that the endpoint of $e$ with 0 at position $i$ is denoted by $a$, and the other endpoint is denoted by $b$. Let $E=\{ f\ |\ f\in E(G) \text{ and } e \Theta f\}$. We claim that $E$ is peripheral with $U_{ba} = W_{ba}$. 

Since $G$ is a partial cube, $W_{ab}$ and $W_{ba}$ partition $V(G)$ into the sets that contain all vertices of $G$ with 0 at position $i$ and all vertices of $G$ with 1 at position $i$, respectively.
Let $x:=b_1 b_2 \ldots b_{i-1} 1 b_{i+1} \ldots b_n \in W_{ba}$ be arbitrarily chosen. Let $y:=b_1 b_2 \ldots b_{i-1} 0 b_{i+1}\ldots b_n$. Clearly $y \in V(G)$, since $y \leq x$, and $y \in W_{ab}$, since it has a 0 at position $i$. The vertices $x$ and $y$ differ in exactly one bit (the one at position $i$), therefore  $xy \in E(G)$. This means that every vertex of $W_{ba}$ is adjacent to a vertex in $W_{ab}$. From definition of $U_{ba}$ it follows that $W_{ba} = U_{ba}$.
\end{proof}

\begin{proposition}\label{everyThetaHas0}
Every $\Theta$-class of a daisy cube $G$ contains an edge with the vertex $0^n$ as an endpoint.
\end{proposition}
\begin{proof}
Let $ab\in E(G)$ be arbitrary, with notation chosen such that the endpoint $a$ has 0 at position $i$ and the endpoint $b$ has 1 at position $i$. Consider the vertices $u=0^n$ and $v=0^{i-1} 1 0^{n-i}$. Since $u\leq b$ and $v \leq b$, both $u$ and $v$ are vertices of $G$. It follows that $uv\in E(G)$. It is trivial to verify that $ab \Theta uv$.
\end{proof}

The following Corollary immediately follows from Proposition \ref{DCisPC} and Proposition \ref{everyThetaHas0}.

\begin{corollary}\label{0isDelta}
If $G$ is a daisy cube, then the vertex $0^n$ is a vertex of degree $\Delta(G)$.
\end{corollary}

\begin{proposition}\label{WisDC}
For any edge $ab$ of a daisy cube $G=Q_n(X)$ the graphs $\langle W_{ab} \rangle$ and $\langle W_{ba} \rangle$ are daisy cubes.
\end{proposition}

\begin{proof}
Let $e=ab$ be an arbitrary edge of $G$, with notation chosen such that the endpoint $a$ has 0 at position $i$ and the endpoint $b$ has 1 at position $i$. Then the set $W_{ab}$ consists of all vertices of $G$ with 0 at position $i$ and $W_{ba}$ consists of all vertices of $G$ with a 1 at position $i$. 

Let $E=\{ f\ |\ f\in E(G) \text{ and } e \Theta f\}$. Using Proposition \ref{thetaIsPeripheral} it follows that $\langle W_{ab} \rangle$ is obtained by contracting the edges of the $\Theta$-class $E$. From Proposition \ref{contractionIsDC} it follows that $\langle W_{ab} \rangle$ is a daisy cube.

To prove that $\langle W_{ba} \rangle$ is a daisy cube consider the following. Let $W'_{ba} = \{ x_1 \ldots x_{i-1} x_{i+1} \ldots x_n  \ |\ x_1\ldots x_n \in W_{ba}\}$. Since all vertices of $W_{ba}$ have 1 at position $i$, it is easy to see that the function $r$ that removes the bit at position $i$ from a binary string of length $n$ is a bijection between $W_{ba}$ and $W'_{ba}$. 

Let $u=u_1 u_2 \ldots u_{n-1} \in W'_{ba}$ be arbitrarily chosen. We claim that any binary string $v$ of length $n-1$, such that $v \leq u$, belongs to $W'_{ba}$. Towards contradiction suppose there is a binary string $w=w_1 w_2 \ldots w_{n-1}$ with $w \leq u$ and $w \not\in W'_{ba}$. Since $u_1 \ldots u_{i-1} 1 u_{i} \ldots u_{n-1} \in W_{ba}$ and $w_1 \ldots w_{i-1} 1 w_{i} \ldots w_{n-1} \leq u_1 \ldots u_{i-1} 1 u_{i} \ldots u_{n-1}$ it follows that $w_1 \ldots w_{i-1} 1 w_{i} \ldots w_{n-1} \in W_{ba}$. But then $w \in W'_{ba}$, a contradiction. Therefore for any $x\in W'_{ba}$ the interval $I_{Q_{n-1}}(x, 0^{n-1}) \subseteq W'_{ba}$. It follows that $\langle W'_{ba} \rangle_{Q_{n-1}}$ is a daisy cube. The assertion follows from the fact that $r$ defines an isomorphism between 
$\langle W_{ba} \rangle_{Q_n}$ and $\langle W'_{ba}\rangle_{Q_{n-1}}$. 
\end{proof}

Let $G$ be a daisy cube. An induced subgraph $H$ of the graph $G$ is called a $\leq$-subgraph if $V(H) = \{u \in V(G) \ |\ u\leq v \text{ for some } v\in V(H)\}$.

The next proposition follows immediately from the definition of a $\leq$-subgraph.

\begin{proposition}\label{leqSubgraphIsDC}
Let $G$ be a daisy cube. If $H$ is a $\leq$-subgraph of $G$, then $H$ is isomorphic to a daisy cube.
\end{proposition}

The converse of Proposition \ref{leqSubgraphIsDC} may not be necessarily true as the example in Figure \ref{figExampleLeq} shows. In Figure \ref{figExampleLeqA} we see the daisy cube $Q_4(X)$, where $X=\{0011, 0110, 1100, 1001\}$. The subgraph $H_1$ induced on the vertex set $\{0010, 0000, 1000\}$ is a $\leq$-subgraph, and is isomorphic to $Q_2(\{01, 10\})$. On the other hand, the subgraph $H_2$ is also isomorphic to $Q_2(\{01, 10\})$, however it is not a $\leq$-subgraph of $Q_4(X)$. Figure \ref{figExampleLeqb} shows the daisy cube $Q_4(Y)$, where $Y=\{0011, 1000, 0100\}$. The subgraph $H$ of $Q_4(Y)$ is isomorphic to the daisy cube $Q_2(\{01, 10\})$ and it is not a $\leq$-subgraph of $Q_4(Y)$.

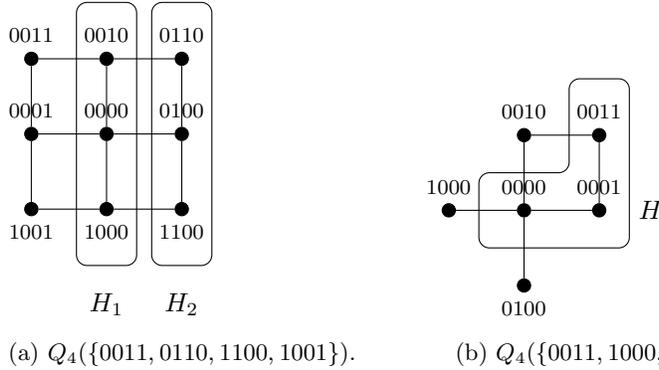
\begin{figure}[!ht]
	\centering
	\begin{subfigure}[t]{0.4\linewidth}
		\begin{tikzpicture}
\tikzstyle{rn}=[circle,fill=black,draw, inner sep=0pt, minimum size=5pt]
\tikzstyle{every node}=[font=\footnotesize]

	  \node [below,font=\normalsize] (g0) at (0, -2) {$H_1$};
	   \node [below,font=\normalsize] (g1) at (1, -2) {$H_2$};
	  
		\node [style=rn] (0) at (0, 0) [label=above:$0000$] {};
		\node [style=rn] (1) at (-1, 0) [label=above:$0001$] {};
		\node [style=rn] (2) at (-1, 1) [label=above:$0011$] {};
		\node [style=rn] (3) at (0, 1) [label=above:$0010$] {};
		\node [style=rn] (4) at (1, 1) [label=above:$0110$] {};
		\node [style=rn] (5) at (1, 0) [label=above:$0100$] {};
		\node [style=rn] (6) at (1, -1) [label=below:$1100$] {};
		\node [style=rn] (7) at (0, -1) [label=below:$1000$] {};
		\node [style=rn] (8) at (-1, -1) [label=below:$1001$] {};
		
		 \draw(1)--(0)--(5);
		 \draw(2)--(3)--(4);
		 \draw(6)--(7)--(8);
		 \draw(6)--(5)--(4);
		 \draw(7)--(0)--(3);
		 \draw(8)--(1)--(2);

		 \draw[rounded corners] (-0.4, 1.75) rectangle (0.4, -1.75) {};
		 \draw[rounded corners] (0.6, 1.75) rectangle (1.4, -1.75) {};
		 
\end{tikzpicture}
\caption{$Q_4(\{0011, 0110, 1100, 1001\})$.}
		\label{figExampleLeqA}
\end{subfigure}\qquad
\begin{subfigure}[t]{0.4\linewidth}
\begin{tikzpicture}
\tikzstyle{rn}=[circle,fill=black,draw, inner sep=0pt, minimum size=5pt]
\tikzstyle{every node}=[font=\footnotesize]  
		\node [style=rn] (0) at (0, 0) [label=above:$0000$] {};
		\node [style=rn] (1) at (-1, 0) [label=above:$1000$] {};

		\node [style=rn] (3) at (0, 1) [label=above:$0010$] {};
		\node [style=rn] (4) at (1, 1) [label=above:$0011$] {};
		\node [style=rn] (5) at (1, 0) [label=above:$0001$] {};
		\node [style=rn] (7) at (0, -1) [label=below:$0100$] {};
		
		 \draw(1)--(0)--(5);
		 \draw(7)--(0)--(3);		 
		 \draw(3)--(4);
		 \draw(5)--(4);	
			 
		 \draw[rounded corners] (0.6, 1.75) -- 
								(1.4, 1.75) -- 
		 						(1.4, -0.5)--
		 						(-0.6, -0.5) --
		 						(-0.6, 0.5) --
		 						(0.6, 0.5)--
		 						cycle {};
		 \node [right,font=\normalsize] (g1) at (1.4, 0) {$H$};
\end{tikzpicture}
		\caption{$Q_4(\{0011, 1000, 0100\})$.}
		\label{figExampleLeqb}
		\end{subfigure}
		\caption{Daisy cubes and their subgraphs.}
		\label{figExampleLeq}
\end{figure}

\begin{corollary}\label{UabIsLeq}
Let $G$ be a daisy cube and $ab\in E(G)$ arbitrary. If the notation is chosen such that $a=a_1 \ldots a_{i-1} 0 a_{i+1} \ldots a_n$ and $b=a_1 \ldots a_{i-1} 1 a_{i+1} \ldots a_n$, then $\langle U_{ab} \rangle$ is a $\leq$-subgraph of $G$.
\end{corollary}

\begin{proof}
Towards contradiction suppose $\langle U_{ab} \rangle$ is not a $\leq$-subgraph. Proposition \ref{thetaIsPeripheral} says that $W_{ba}=U_{ba}$. Then for some vertex of $x=x_1\ldots x_n \in U_{ab}$ there is a vertex $y=y_1 \ldots y_n \in \Bn$ such that $y\leq x$ and $y \not \in U_{ab}$. Note that $x_i=0$ and since $y\leq x$, also $y_i=0$. Since all vertices in relation $\leq$ with $x$ are in $G$ (by definition of a daisy cube), then $y\in W_{ab} \setminus U_{ab}$. Let $x' = x_1 \ldots x_{i-1} 1 x_{i+1} \ldots x_n$ and $y' = y_1 \ldots y_{i-1} 1 y_{i+1} \ldots y_n$. Since $x\in U_{ab}$, and the set $U_{ba}$ contains only vertices with 1 at position $i$, and $x$ and $x'$ differ in exactly one position, it follows that $x' \in U_{ba}$. The fact $y \leq x$ implies that $y'\leq x'$, moreover $y'\in U_{ba}$, since $y_i=1$. The vertices $y$ and $y'$ differ in exactly one position, which means that $yy'\in E(G)$, further implying $y\in U_{ab}$, a contradiction.
\end{proof}

Let $H$ be a $\leq$-subgraph of a daisy cube $G$. Then the peripheral expansion $\pe(G;V(H))$ is called the $\leq$-expansion of $G$ with respect to $H$.

\begin{proposition}\label{leqExIsDaisy}
Let $G=Q_h(X)$ be a daisy cube and $H$ a $\leq$-subgraph of $G$. The $\leq$-expansion of $G$ with respect to $H$ is a daisy cube.
\end{proposition}

\begin{proof}
Let $G'=\pe(G;V(H))$. Then $G'$ consists of a disjoint union of a copy of $G$ and a copy $H$ with an edge between each vertex of $H$ and the corresponding vertex in the copy of $H$. Let the labels of the vertices of $G'$ be defined as follows. Prepend a 0 to the label of each vertex in $G'$ corresponding to the copy of $G$ and a 1 to each vertex of $G'$ corresponding to the copy of $H$. So the labels of the vertices of $G'$ are binary strings of length $h+1$.  We prove the assertion in two steps: showing first, that the vertex set of $G'$ is the vertex set of a daisy cube, and concluding the proof by showing that two vertices are adjacent if and only if they differ in exactly one bit.

First, we will show that for every $v \in V(G')$ all the vertices in relation $\leq$ with $v$ are also in $V(G')$. Let $v\in V(G')$ be arbitrary. If $v=0 v_1 \ldots v_h$ ($v$ is in the copy of $G$ in $G'$) then any binary string $u$ of length $h+1$,  such that $u\leq v$, is also of the form $0 u_1 \ldots u_h$. Since all binary strings $w_1\ldots w_h$ of length $h$ such that $w_1\ldots w_h \leq v_1\ldots v_h$ are in $G$, the assertion for this case follows. Now, let $v=1 v_1 \ldots v_h$ ($v$ is in the copy of $H$ in $G'$). Since $H$ is a $\leq$-subgraph, following the same line of thought as in the previous case, we can show that all binary strings $u=1 u_1 \ldots u_h$ such that $u\leq v$, are also in $V(G')$. To see, that also all binary strings $u=0 u_1 \ldots u_h$ such that $u\leq v$ are in $V(G')$, remember that there is and edge between $v$ (which is in the copy of $H$) and a vertex $w$ in the copy of $G$. Since $v=1 v_1 \ldots v_h$ and $w=0 w_1 \ldots w_h$ and $vw \in E(G')$ it follows that for every $i\in\{1, \ldots, h\}$ the values $v_i$ and $w_i$ are equal. From the first part, all vertices in relation $\leq$ with $w$ are in $V(G')$ and these are exactly the vertices $u = 0 u_1 \ldots u_h$ such that $u\leq v$.

Second, we show that two vertices are adjacent if they differ in exactly one position. Take two arbitrary adjacent vertices of $G'$, say $v=v_0 v_1 \ldots v_h$ and $u=u_0 u_1 \ldots u_h$. If both $v_0$ and $u_0$ are 0 (or 1), then $u$ and $v$ correspond to two adjacent vertices in the copy of $G$ (or $H$), which is an induced subgraph of $Q_h$ (is a $\leq$-subgraph), and the assertion follows.  If $0 = v_0 \not = u_0 = 1$, then there is an edge between the two vertices if and only if $v_1 \ldots v_h$ is a vertex of $H$ in $G$ and $u_1 \ldots u_h$ is the corresponding vertex in the copy of $H$, meaning that for every $i\in\{1, \ldots, h\}$ it holds that $v_i = u_i$. This completes the proof.
\end{proof}

Using Proposition \ref{UabIsLeq} and Proposition \ref{leqExIsDaisy} we immediately obtain the following characterization of daisy cubes.

\begin{theorem}\label{characterizationDCexp}
A connected graph $G$ is a daisy cube if and only if it can be obtained from the one-vertex graph by a sequence of $\leq$-expansions.
\end{theorem}

From this characterization and Theorem \ref{muldersExpThm} we can also characterize all median daisy cubes as follows.

\begin{corollary}
A daisy cube is a median graph if and only if it can be obtained from the one-vertex graph by a sequence of convex $\leq$-expansions.
\end{corollary}

\section{Finding proper labellings}\label{findingLabellings}
In this section we present an algorithm which for a given unlabelled graph $G$ isomorphic to a daisy cube (the embedding into the corresponding hypercube is not given) assigns a proper labelling to vertices of $G$.

\begin{proposition}\label{0isLeft}
Let $G=Q_n(X)$ be a daisy cube and $e=ab \in E(G)$. If the notation of the endpoints of $e$ can be chosen such that $|W_{ab}| > |W_{ba}|$, then $0^n \in W_{ab}$.
\end{proposition}

\begin{proof}
By Corollary \ref{UabIsLeq}, the notation can be chosen such that $a$ has a 0 at position $i$ and $b$ has a 1 at position $i$ and $U_{ab}$ induces a $\leq$-subgraph. For any vertex $u\in V(G)$ it holds that $0^n \leq u$ and therefore $0^n \in U_{ab} \subseteq W_{ab}$.
\end{proof}

\begin{proposition}\label{0canBeBoth}
Let $G=Q_n(X)$ be a daisy cube and $e=ab \in E(G)$. If $|W_{ab}| = |W_{ba}|$, then there exists a proper labelling of $G$ such that $0^n \in W_{ab}$ and a proper labelling of $G$ such that $0^n \in W_{ba}$.
\end{proposition}

\begin{proof}
The fact that $|W_{ab}| = |W_{ba}|$ and that every $\Theta$-class in $G$ is peripheral (Proposition \ref{thetaIsPeripheral}) implies that $W_{ab} = U_{ab}$ and $W_{ba} = U_{ba}$, moreover the $\Theta$-equivalence class of $ab$ induces a matching (and also an isomorphism of corresponding induced subgraphs) between $U_{ab}$ and $U_{ba}$. Therefore $G$ is isomorphic to $\langle U_{ab} \rangle_G \Box K_2$. By Corollary \ref{UabIsLeq} at least one of $U_{ab}$ and $U_{ba}$ induces a $\leq$-subgraph. W.l.o.g., assume it is $U_{ab}$, therefore $0^n \in U_{ab}$. Let $u \in U_{ba}$ be the vertex adjacent to $0^n$. It follows that $u$ has exactly one 1, say at position $i$. Moreover, all vertices in $U_{ab}$ have 0 at position $i$ and all vertices of $U_{ba}$ have 1 at position $i$. Changing values at position $i$ to 1 for all vertices of $U_{ab}$ and to 0 for all vertices of $U_{ba}$ we obtain another labelling of the vertices of $G$. It is easy to verify that such labelling is a proper labelling and that in this case $0^n$ belongs to $U_{ba}$ and therefore to $W_{ba}$.
\end{proof}

From Proposition \ref{0isLeft} and Proposition \ref{0canBeBoth} we obtain the following corollary.

\begin{corollary}\label{0jeVVecjem}
Let $G=Q_n(X)$ be a daisy cube and $e=ab \in E(G)$. If the notation can be chosen such that $|W_{ab}| \geq |W_{ba}|$, then there exists a proper labelling of $G$ such that $0^n \in W_{ab}$.
\end{corollary}

These results enable us to find a proper labelling for any graph isomorphic to a daisy cube as presented in Algorithm \ref{algOznake}.

\begin{algorithm}[!ht]
\KwIn{an unlabelled graph $G$ isomorphic to a daisy cube}
\KwOut{a proper labelling of $G$}
\DontPrintSemicolon
\BlankLine
Compute $\Theta$ and denote $\Theta$-classes by $\Theta_1, \ldots, \Theta_k$.\;
\For{$i=1$ \KwTo $k$} {
	Choose $ab \in \Theta_i$ arbitrarily. \;
	Determine $W_{ab}$ and $W_{ba}$. \;
	\eIf{$|W_{ab}| \geq |W_{ba}|$}{
		$W' := W_{ab}$ \;
		$W'' := W_{ba}$ \;
	}{
		$W' := W_{ba}$ \;
		$W'' := W_{ab}$ \;
	}
	\ForAll{$v \in W'$}{
		Set the $i$th coordinate of the label of $v$ to 0. \;
	}
	\ForAll{$v \in W''$}{
		Set the $i$th coordinate of the label of $v$ to 1. \;
	}
}
\caption{Proper labelling of a daisy cube}\label{algOznake}
\end{algorithm}

\begin{theorem}
Algorithm \ref{algOznake} assigns a proper labelling to vertices of an unlabelled graph isomorphic to a daisy cube in $O(mn)$ time, where $n$ is the number of vertices and $m$ the number of edges of $G$.
\end{theorem}

\begin{proof}
The correctness of the assigned labels follows from Proposition \ref{0isLeft} and Corollary \ref{0jeVVecjem}. 

Now consider the time complexity. Step 1 can be done in $O(mn)$ time \cite[Theorem 18.6]{HIK2011knjiga}. For one iteration of the loop at step 2 the following holds. The edge $ab$ can be chosen in constant time. The sets $W_{ab}$ and $W_{ba}$ can be computed in $O(m)$ time (e.g. by using BFS from each endpoint of $ab$). Both loops (step 12 and step 15) run together in $O(n)$ time. Since the number of $\Theta$-classes (the number of times the for loop at step 2 repeats) is bounded by $n$, the assertion follows.
\end{proof}

\section{Generalization of daisy cubes}
In this section we give a generalization of the concept of daisy cubes. All graphs considered are connected.

\providecommand{\lgr}{\leq_{G,r}}
\providecommand{\lqz}{\leq_{Q_h,0^h}}
\begin{definition}\label{defLgr}
Let $G$ be a rooted graph with the root $r$. Let $u$ and $v$ be two vertices of $G$. We say that $u \lgr v$, if $u$ is on some shortest $v,r$-path.
\end{definition}

\begin{proposition}
Let $G$ be a rooted graph with the root $r$. The relation $\lgr$ is a partial order on $V(G)$.
\end{proposition}

\begin{proof}
Since every vertex $v \in V(G)$ is on every shortest path from $v$ to $r$, it follows that $\lgr$ is reflexive. 

Assume that for any two vertices $u$ and $v$ from $V(G)$ it holds that $u\lgr v$ and $v \lgr u$. This means that $u$ is on some shortest $v,r$-path and that $v$ is on some shortest $u,r$-path. Therefore, $d(v,r)=d(v,u) + d(u,r)$ and $d(u,r)=d(u,v)+d(v,r)$. This implies that $d(u,v)=0$, meaning that $u=v$. Therefore $\lgr$ is antisymmetric.

Let $u, v$ and $w$ be arbitrary vertices of $G$. Also, let $u\lgr v$ and $v \lgr w$. This means $d(v,r)=d(v,u) + d(u,r)$ and that $d(w,r)=d(w,v)+d(v,r)$. This gives that $d(w,r)=d(w,v)+d(v,u)+d(u,r)$ which implies that $u$ is on some shortest $w,r$-path and therefore $u\lgr w$. This proves the transitivity property and concludes the proof.
\end{proof}

\begin{definition}\label{defDG}
Let $G$ be a rooted graph with the root $r$. For $X\subseteq V(G)$ the \emph{daisy graph $G_r(X)$ of the graph $G$ with respect to $r$ (generated by $X$)} is the subgraph of $G$ where $$G_r(X)=\left\langle \{ u \in V(G)\ |\ u\lgr v \text{ for some }v\in X\} \right\rangle.$$
\end{definition}

From Definition \ref{defDG} it immediately follows that if $u$ is a vertex of the daisy graph $G_r(X)$ of the graph $G$ with respect to $r$, then $I_G(u,r) \subseteq V(G_r(X))$. This fact also immediately gives the following result.

\begin{proposition}
Let $G$ be a rooted graph with the root vertex $r$. If $H$ is a convex subgraph of $G$, such that $r \in V(H)$, then $H$ is a daisy graph of $G$ with respect to $r$.
\end{proposition}

By definition of relation $\leq$ on the vertices of the hypercube $Q_h$ (binary strings of length $h$), one immediately obtains that for two such vertices, say $u$ and $v$, the following is true: $u \leq v$ if and only if $u$ is on some shortest path between $v$ and $0^h$, which is definition of $\lqz$. Therefore daisy cubes are a special case of daisy graphs, specifically as the following proposition says. 

\begin{proposition}
If $G=Q_h(X)$ is a daisy cube, then $G$ is a daisy graph of $Q_h$ with respect to $0^h$.
\end{proposition}

As stated in Proposition \ref{DCisPC}, daisy cubes are isometric subgraphs of hypercubes. We now give a sufficient condition for when a daisy graph of a rooted graph $G$ with respect to the root $r$ is isometric.

\begin{proposition}\label{medImpliesIsometric}
Let $G$ be a rooted graph with the root $r$. If for any two vertices of $G$, say $u$ and $v$, it holds that there exists a median of $u$, $v$ and $r$, then every daisy graph of $G$ with respect to $r$ is isometric in $G$.
\end{proposition}

\begin{proof}
Let $H$ be an arbitrary daisy graph of $G$ with respect to $r$. Also, let $u$ and $v$ be two arbitrary vertices of $H$, and let $w$ be a median of $u$, $v$ and $r$. Since $I_G(u,r) \subseteq V(H)$, $I_G(v,r) \subseteq V(H)$ and $H$ is an induced subgraph of $G$, then any shortest $u,r$-path and any shortest $v,r$-path is completely contained in $H$. Also, $w$ is on some shortest $u,r$-path, as well as on some shortest $v,r$-path. This implies that $d_G(u,w) = d_H(u,w)$ and $d_G(v,w) = d_H(v,w)$. We know that $w$ is also on some shortest $u,v$-path in $G$, implying $d_G(u,v) = d_G(u,w) + d_G(w,v)$. Also,  $d_H(u,v) \leq d_H(u,w) + d_H(w,v) = d_G(u,w) + d_G(w,v) = d_G(u,v)$. Since $H$ is a subgraph of $G$ we also know that $d_H(u,v)\geq d_G(u,v)$, thus $d_H(u,v)=d_G(u,v)$. Since this holds for two arbitrary vertices $u$ and $v$, it follows that $H$ is an isometric subgraph of $G$.
\end{proof}

The converse of Proposition \ref{medImpliesIsometric} may not necessarily be true. In Figure \ref{figConvMedImplIso} we see a rooted graph $G$ with the root $r$ and the daisy graph $H$ of the graph $G$ with respect to $r$ generated by $\{u,v\}$. The edges of $H$ are depicted with thick lines. $H$ is an isometric subgraph of $G$, however, no median exists for the vertices $u$, $v$ and $r$.

\begin{figure}[!ht]
	\centering
	\begin{tikzpicture}
\tikzstyle{rn}=[circle,fill=black,draw, inner sep=0pt, minimum size=5pt]
\tikzstyle{every node}=[font=\footnotesize]
		\node [style=rn] (0) at (-2,2) [label=above:$u$] {};  		
		\node [style=rn] (1) at (-0.6666666,2) [] {};
		\node [style=rn] (2) at (0.6666666,2) [] {};
		\node [style=rn] (3) at (2,2) [label=above:$v$] {};
		\node [style=rn] (4) at (-1,1) [] {};
		\node [style=rn] (5) at (1,1) [] {};
		\node [style=rn] (6) at (0,0) [label=below:$r$] {};
		
		\draw (0)--(1)--(2)--(3)--(5)--(4)--(0);
		\draw (4)--(6)--(5);
		
		\draw[line width=3pt] (0)--(4)--(5)--(3);
		\draw[line width=3pt] (4)--(6)--(5);
		
		\draw[rounded corners] (-2.4, 2.75) rectangle (2.4, -0.75) {};
		\node [left,font=\normalsize] (g0) at (-2.4, 1) {$G$};
\end{tikzpicture}
		\caption{A counter example for the converse of Proposition \ref{medImpliesIsometric}.}
		\label{figConvMedImplIso}
\end{figure}
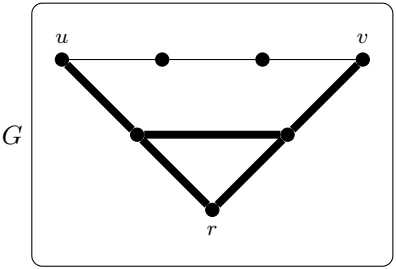

\begin{corollary}
Let $G$ be a rooted median graph with the root $r$. Every daisy graph of $G$ with respect to $r$ is isometric in $G$.
\end{corollary}

\begin{proof}
Since in median graphs every triple of vertices admits a (unique) median, then for two arbitrary vertices of $G$ and the root $r$ there also exists a median. By Proposition \ref{medImpliesIsometric} the assertion follows.
\end{proof}

\section{Conclusion}
In this paper we answered an open problem by Klav\v zar and Mollard proposed in \cite{KlaMol-18}. Namely, we provide a characterization of daisy cubes in terms of an expansion procedure. We also presented some interesting properties of daisy cubes and provide an $O(mn)$ time algorithm for finding a proper embedding of a daisy cube into the corresponding hypercube. Also, we proposed a generalization of daisy cubes and  provided some exciting results concerning these graphs.

Moreover, the following further investigations might be interesting.
\begin{problem}
Find a non-constructive characterization of daisy cubes.
\end{problem}

\begin{problem}\label{problem2}
Is there a faster way of finding the vertex $0^h$ of a daisy cube $Q_h(X)$ than the one provided in Algorithm \ref{algOznake}?
\end{problem}

A positive answer to Problem \ref{problem2} would give a linear time algorithm for finding a proper labelling of a graph isomorphic to a daisy cube.

\begin{problem}
Provide a recognition algorithm for daisy cubes.
\end{problem}

Proposition \ref{medImpliesIsometric} provides a condition, such that if a rooted graph $G$ with the root $r$ satisfies this condition, then every daisy graph of the graph $G$ with respect to $r$ is isometric.

\begin{problem}
Characterize rooted graphs $G$ with the root $r$ for which every daisy graph of a rooted graph $G$ with respect to $r$ is isometric.
\end{problem}

\subsection*{Acknowledgments}
This work was supported by the Slovenian Research Agency under the grants P1-0297 and J1-9109.


\begin{thebibliography}{10}
\bibitem{chepoi-88}
{\sc V. D. Chepoi}, {\em Isometric subgraphs of Hamming graphs and d-Convexity}, Cybernetics 24(1) (1988), 6-–9, https://doi.org/10.1007/BF01069520


\bibitem{HIK2011knjiga}
{\sc R.~Hammack, W.~Imrich, S.~Klav\v{z}ar}: {\em Handbook of Product
  Graphs, second edition}. CRC Press, Boca Raton, 2011.

\bibitem{K2012}
{\sc S.~Klav\v{z}ar}: {\em Structure of fibonacci cubes: a survey}. J. Comb. Optim., \textbf{25} (2013),  505--522, https://doi.org/10.1007/s10878-011-9433-z.

\bibitem{KlaMol-18}
{\sc S.~Klav\v{z}ar, M. Mollard}: {\em Daisy cubes and distance cube polynomial}, European Journal of
Combinatorics (2018), https://doi.org/10.1016/j.ejc.2018.02.019.

\bibitem{MCS2001}
{\sc E.~Munarini, C.~P. Cippo, N.~Z. Salvi}: {\em On the lucas cubes}.  Fibonacci Quart., \textbf{39} (1) (2001), 12--21.

\bibitem{T2013}
{\sc A.~Taranenko.}: 
\emph{A new characterization and a recognition algorithm of Lucas cubes.}
Discrete Math. Theor. Comput. Sci., \textbf{15} (2013), 31--39

\bibitem{V2019}
{\sc A.~Vesel}: {\em Cube-complements of generalized Fibonacci cubes}. Discrete Math., \textbf{342} (2019), 1139--1146. https://doi.org/10.1016/j.disc.2019.01.008

\bibitem{W1984}
{\sc P.~M. Winkler}: {\em Isometric embeddings in products of complete graphs}.
  Discrete Appl. Math., \textbf{7} (2) (1984), 221--225.

\bibitem{ZP2018}
{\sc P.~{\v{Z}}igert Pleter\v{s}ek}, {\em Resonance Graphs of Kinky Benzenoid Systems Are Daisy Cubes}. MATCH Commun. Math. Comput. Chem., \textbf{80} (2018),
 207--214.
\end{thebibliography}
\end{document}